\documentclass[11pt]{article}
\usepackage{amsmath, amssymb, amsfonts, amsthm, amscd, vmargin}
\usepackage[latin1]{inputenc}

\usepackage[all]{xy}
\usepackage{graphicx,color}
\usepackage{enumerate}

\usepackage{tikz}
\usetikzlibrary{shapes}
\usetikzlibrary{matrix,arrows,decorations.pathmorphing}

\setmarginsrb{2.4cm}{4.5cm}{2.4cm}{3.3cm}{0cm}{0cm}{0cm}{1.5cm}

\theoremstyle{plain}
\newtheorem{teo}{Theorem}
\newtheorem{lem}[teo]{Lemma}
\newtheorem{prop}[teo]{Proposition}
\newtheorem{cor}[teo]{Corollary}

\theoremstyle{definition}
\newtheorem{defi}{Definition}

\newtheorem{rem}{Remark}

\renewcommand{\Im}{\operatorname{Im}}

\newcommand{\Cbb}{{\mathbb C}}
\newcommand{\Qbb}{{\mathbb Q}}
\newcommand{\Zbb}{{\mathbb Z}}

\newcommand{\bigslant}[2]{{\raisebox{.4em}{$#1$}\left/\raisebox{-.8em}{$#2$}\right.}}

\begin{document}

\title{A Minimal Model Program for $\Qbb$-Gorenstein varieties}

\author{Boris Pasquier\protect\footnote{Boris PASQUIER, 
E-mail: boris.pasquier@univ-montp2.fr}}

\maketitle
\begin{abstract}
The main results of this paper are already known \cite{Shokurov}. Moreover, the non-$\mathbb{Q}$-factorial MMP was more recently considered by O~Fujino, in the case of toric varieties \cite{Fujino1}, for klt pairs \cite{Fujino2} and more generally for log-canonical pairs \cite{Fujino3}.

Here we rewrite the proofs of some of these results, by following the proofs given in \cite{KMM} of the same results in $\mathbb{Q}$-factorial MMP. And, in the family of $\mathbb{Q}$-Gorenstein spherical varieties, we answer positively to the questions of existence of flips and of finiteness of sequences of flips.

I apologize for the first version of this paper, which I wrote without knowing that these results already exist.
\end{abstract}

\\


\section{Introduction}

We only consider normal algebraic varieties over $\Cbb$.

In this paper we prove some already well-known results on $\Qbb$-Gorenstein MMP (also called non-$\mathbb{Q}$-factorial MMP).
The motivation comes from \cite{MMPhoro}, where the MMP for horospherical varieties is reduced to the study of a family of moment polytopes, $\Qbb$-Gorenstein MMP seams to be more natural than $\Qbb$-factorial MMP.

Then we answer positively to the questions of existence of flips and of finiteness of sequences of flips for spherical varieties.

Before to describe explicitly the results, we recall some basic definitions and notations.

\begin{defi}
\begin{itemize}
\item A normal variety $X$ is $\Qbb$-Gorenstein if its canonical divisor $K_X$ is $\Qbb$-Cartier (ie if there exists a multiple of $K_X$ that is Cartier).
\item A normal $\Qbb$-Gorenstein variety $X$ has terminal singularities if there exists a desingularization $\sigma:\,V\longrightarrow X$ of $X$ such that $$K_V=\sigma^*K_X+\sum_{E\,{\rm prime}\,{\rm divisor}\,{\rm of }\,V}a_EE,$$
with $a_E>0$ for any exceptional prime divisor $E$ of $\sigma$ (and $a_E=0$ otherwise).
\item A contraction (ie a projective morphism $\phi:\,X\longrightarrow Y$ such that $\phi_*\mathcal{O}_X=\mathcal{O}_Y$) is a small contraction if its exceptional locus is at least of codimension~2.
\end{itemize}
\end{defi}

Now, fix a normal $\Qbb$-Gorenstein projective variety $X$ with terminal singularities.
We denote by $NE(X)$ the nef cone of curves of $X$ and by $NE(X)_{K_X<0}$ (resp. $NE(X)_{K_X>0}$) the intersection of the nef cone with the open half-space of curves negative (resp. positive) along the divisor $K_X$. Then, by the Cone Theorem, $NE(X)=NE(X)_{K_X>0}+\sum_{i\in I}R_i$, where the set $\{R_i\,\mid\,i\in I\}$ is a discrete set of (extremal) rays of $NE(X)_{K_X<0}$. Fix a ray $R$ of $NE(X)_{K_X<0}$. Then, by the Contraction Theorem, there exists a unique contraction $\phi:\,X\longrightarrow Y$, such that, for any curve $C$ of $X$, $\phi(C)$ is a point if and only if the class of $C$ is in $R$.

We recall the definition of flips we use in this paper.
\begin{defi}\label{def:QGorflip}
Let $X$ be a normal $\Qbb$-Gorenstein projective variety, and $\phi:X\longrightarrow Y$ be a birational contraction of a ray of $NE(X)_{K_X<0}$, such that $Y$ is not $\Qbb$-Gorenstein.

A flip of $\phi$ is a small contraction $\phi:X^+\longrightarrow Y$, where 
\begin{itemize}
\item $X^+$ is a normal $\Qbb$-Gorenstein projective variety;
\item and for any contracted curve $C$ of $X^+$, $K_{X^+}\cdot C>0$.
\end{itemize} 
\end{defi}

In this paper, we prove the following (already known \cite{Shokurov}) result.

\begin{teo}\label{th:main} Let $X$ be a normal $\Qbb$-Gorenstein projective variety with terminal singularities. Let $R$ be a ray of $NE(X)_{K_X<0}$ and denote by $\phi:\,X\longrightarrow Y$ the contraction of $R$. Suppose that $\phi$ is birational.
\begin{itemize}
\item If $Y$ is $\Qbb$-Gorenstein, then $Y$ has terminal singularities and $\phi$ contracts a $\Qbb$-Cartier divisor.
\item If $Y$ is not $\Qbb$-Gorenstein, then there exists a flip of $\phi$ if and only if $\oplus_{l\geq 0}\phi_*\mathcal{O}_X(lmK_X)$ is a finitely generated sheaf of $\mathcal{O}_Y$-algebras. Moreover, if it exists, it is unique and $X^+$ has terminal singularities. 
\end{itemize}
\end{teo}

The proof (given in this paper) of Theorem~\ref{th:main} is inspired by the proofs of the same results in the original MMP \cite{KMM}, which are also detailed in \cite{matsuki}.\\

The paper is organized as follows.

In Section~\ref{section1}, we list several probably well-known and general results, which are very useful in the rest of the paper.

Sections~\ref{section2} and~\ref{section3} are devoted to the proof of Theorem~\ref{th:main}.

In Section~\ref{section4}, we illustrate the MMP for $\Qbb$-Gorenstein projective varieties in an example of a 3-dimensional toric variety. In particular, we observe that a flip is not necessarily a contraction of an extremal ray.

In Section~\ref{section5}, we explain how to run a $\Qbb$-Gorenstein log MMP.

And in Section~\ref{section6} we answer positively to the questions of existence of flips and of finiteness of sequences of flips.

\section{General lemmas}\label{section1}
	
We begin by a very classical result.
	
\begin{lem}\label{lem:eff}
Let $X$ be a normal variety, and let $D$ be a Cartier divisor of $X$.
Then the following assertions are equivalent.
\begin{enumerate}
\item $D$ is effective.
\item $\mathcal{O}_X\subset \mathcal{O}_X(D)$.
\item $\mathcal{O}_X(-D)\subset \mathcal{O}_X$.
\end{enumerate}
\end{lem}
	
%
%

From Lemma~\ref{lem:eff}, we deduce the following, probably well-known, result.

\begin{lem}\label{lem:pulleff}
Let $\phi:\,X\longrightarrow Y$ be a morphism between two normal varieties $X$ and $Y$.
Let $D$ be an effective $\Qbb$-Cartier divisor of $Y$.

Then $\phi^*D$ is an effective $\Qbb$-Cartier divisor of $X$.
\end{lem}	

\begin{proof}
Let $m$ be a positive integer such that $mD$ is Cartier. Since $mD$ is effective, we have $\mathcal{O}_Y\subset \mathcal{O}_Y(mD)$. Then the image of $\phi^*(\mathcal{O}_Y)$ in $\phi^*(\mathcal{O}_Y(mD))$, is the pull-back by $\phi$ of the image of $\mathcal{O}_Y$ in $\mathcal{O}_Y(mD)$. But all these invertible sheaves are subsheaves of the constant sheaf $\mathcal{K}=\Cbb(X)$. Thus, $\phi^*(\mathcal{O}_Y)=\mathcal{O}_X$ is contained in $\phi^*(\mathcal{O}_Y(mD))$. Then $\phi^*(mD)$  is effective and also is $\phi^*(D):=\frac{1}{m}\phi^*(mD)$.
\end{proof}

We know prove the key lemma of the paper.

\begin{lem}\label{lem:key}
Let $\phi:\,X\longrightarrow Y$ be a surjective birational morphism between two normal varieties $X$ and $Y$.
Let $D$ be a Cartier divisor of $X$ and denote by $D_Y$ the Weil divisor $\phi_*D$.

Suppose that the morphism $\phi^*\phi_*\mathcal{O}_X(-D)\longrightarrow\mathcal{O}_X(-D)$ is surjective.

Then the image of $\phi^*\mathcal{O}_Y(D_Y)\longrightarrow\mathcal{K}$ is contained in $\mathcal{O}_X(D)$.

In particular, $\phi_*\mathcal{O}_X(D)=\mathcal{O}_Y(D_Y)$.
\end{lem}

Note that $D_Y$ is not necessarily Cartier, and $\mathcal{O}_Y(D_Y)$ is the (not necessarily invertible) subsheaf of the constant sheaf $\mathcal{K}=\Cbb(X)=\Cbb(Y)$ defined by, for any open set $\mathcal{U}$ of $Y$, $$\mathcal{O}_Y(D_Y)(\mathcal{U})=\{f\in\Cbb(Y)\,\mid\,{\rm div}\,f_{|\mathcal{U}}+D_{Y|\mathcal{U}}\geq 0\}.$$

\begin{proof}
Consider the inclusion of sheaves of $\mathcal{O}_Y$-modules $\phi_*\mathcal{O}_X(-D)\subset\mathcal{O}_Y(-D_Y)$. Tensoring by $\mathcal{O}_Y(D_Y)$, and composing by the natural morphism $\mathcal{O}_Y(D_Y)\otimes_{\mathcal{O}_Y}\mathcal{O}_Y(-D_Y)\longrightarrow\mathcal{O}_Y$, we get the natural morphism of sheaves of $\mathcal{O}_Y$-modules:
$$\mathcal{O}_Y(D_Y)\otimes_{\mathcal{O}_Y}\phi_*\mathcal{O}_X(-D)\longrightarrow\mathcal{O}_Y.$$

The pull-back of this morphism $$\phi^*\mathcal{O}_Y(D_Y)\otimes_{\mathcal{O}_X}\phi^*\phi_*\mathcal{O}_X(-D)\longrightarrow\mathcal{O}_X$$ factors through the following surjective morphism $$\phi^*\mathcal{O}_Y(D_Y)\otimes_{\mathcal{O}_X}\phi^*\phi_*\mathcal{O}_X(-D)\longrightarrow\phi^*\mathcal{O}_Y(D_Y)\otimes_{\mathcal{O}_X}\mathcal{O}_X(-D).$$

Then the image of $\phi^*\mathcal{O}_Y(D_Y)\otimes_{\mathcal{O}_X}\mathcal{O}_X(-D)$ in $\mathcal{K}$ is contained in $\mathcal{O}_X$. We tensor by the invertible sheaf $\mathcal{O}_X(D)$, and we get that $\phi^*\mathcal{O}_Y(D_Y)$ maps to $\mathcal{O}_X(D)$.

To prove the last statement we use that, since $\phi$ is surjective and $\phi_*\mathcal{O}_X=\mathcal{O}_Y$, for any sheaf of $\mathcal{O}_Y$-modules $\mathcal{F}$, we have $\phi_*\phi^*\mathcal{F}=\mathcal{F}$.
Hence, $\mathcal{O}_Y(D_Y)=\phi_*\phi^*\mathcal{O}_Y(D_Y)$ maps to $\phi_*\mathcal{O}_X(D)$. Both are subsheaves of $\mathcal{K}$, so we deduce that this map is an inclusion. The other inclusion is obvious, so that $\phi_*\mathcal{O}_X(D)=\mathcal{O}_Y(D_Y)$.
\end{proof}

From these lemmas, we get a useful corollary.

\begin{cor}\label{cor:keyeff}
Let $\phi:\,X\longrightarrow Y$ be a surjective birational morphism between two normal varieties $X$ and $Y$.
Suppose that there exists a positive integer $m$ such that $mK_X$ and $mK_Y$ are Cartier and such that the morphism $\phi^*\phi_*\mathcal{O}_X(-mK_X)\longrightarrow\mathcal{O}_X(-mK_X)$ is surjective.

Then $K_X-\phi^*K_Y$ is effective.
\end{cor}

\begin{proof} We apply Lemma~\ref{lem:key} to $D=mK_X$ (and $D_Y=mK_Y$). 
Since $mK_Y$ is Cartier, then $\phi^*\mathcal{O}_Y(mK_Y)$ is a subsheaf of $\mathcal{K}$ and the map $\phi^*\mathcal{O}_Y(mK_Y)\longrightarrow\mathcal{O}_X(mK_X)$ is an inclusion. We deduce that $\mathcal{O}_X\subset\mathcal{O}_X(mK_X)\otimes_{\mathcal{O}_X}\phi^*\mathcal{O}_Y(-mK_Y)$. And we apply Lemma~\ref{lem:eff} to conclude. 
\end{proof}

%
%

\section{$\Qbb$-Gorenstein divisorial contractions}\label{section2}

In this section we study the contractions that play the role of divisorial contractions.

\begin{defi}
The contraction $\phi:\,X\longrightarrow Y$ is called a $\Qbb$-Gorenstein divisorial contraction if it is birational and $Y$ is $\Qbb$-Gorenstein.
\end{defi}

The aim of this section is to prove that $Y$ has the same singularities as $X$, and that $\phi$ contracts a (not necessarily irreducible) Cartier divisor.

\begin{teo}\label{th:QGorDiv}\cite{Shokurov}
Let $X$ be a normal $\Qbb$-Gorenstein projective variety with terminal singularities. Let $R$ be a ray of $NE(X)_{K_X<0}$ and denote by $\phi:\,X\longrightarrow Y$ the contraction of $R$.

If $\phi$ is a $\Qbb$-Gorenstein divisorial contraction, then $Y$ is $\Qbb$-Gorenstein with terminal singularities and $\phi$ contracts the (not zero) Cartier divisor $E:=K_X-\phi^*K_Y$.
\end{teo}	

\begin{proof}
By hypothesis, $E$ is a $\Qbb$-Cartier divisor of $X$ such that $E\cdot C<0$ for any curve $C$ of class in $R$. In particular $E$ is not zero. Write $E=\sum_{i\in I}a_iE_i$ where for any $i\in I$, where $\{E_i\,\mid\,i\in I\}$ is the set of irreducible exceptional divisors of $\phi$.

Let $\sigma:\,V\longrightarrow X$ be a desingularization of $X$. Denote by $F_j$, with $j\in J$ the irreducible exceptional divisors of $\sigma$, and for any $i\in I$, denote by $G_i$ the strict transform of $E_i$ by $\sigma$. Write $K_V-\sigma^*K_X=\sum_{j\in J}b_jF_j$ and $\sigma^*(K_X-\phi^*K_Y)=\sum_{j\in J} c_jF_j+\sum_{i\in I}a_iG_i.$ since $X$ has terminal singularities, for any $j\in J$, the rational numbers $b_j$ are positive. And, by Lemma~\ref{lem:pulleff} and Corollary~\ref{cor:keyeff}, for any $i\in I$ and any $j\in J$, the rational numbers $a_i$  and $c_j$ are non-negative.

Hence, it is now enough to prove that the rational numbers $a_i$ are positive (or non-zero). Let $i_0\in I$. There exists a curve $C$ in $V$ that is contracted by $\phi\circ\sigma$, contained in $G_{i_0}$ but not in $\rm{Exc}(\sigma)\cup\bigcup_{i\in I,\, i\neq i_0}G_i$. In particular, $\sigma(C)$ is a curve of $X$ that is contracted by $\phi$, for any $j\in J$ we have $F_j\cdot C\geq 0$, and for any $i\in I,\, i\neq i_0$ we have $G_i\cdot C\geq 0$. Then, on the one hand $(K_X-\phi^*K_Y)\cdot\sigma(C)=K_X\cdot\sigma(C)$ so that $\sigma^*(K_X-\phi^*K_Y)\cdot C<0$, and on the other hand $$\sigma^*(K_X-\phi^*K_Y)\cdot C=\sum_{i\in I}a_i(G_i\cdot C)+\sum_{j\in J}c_j(F_j\cdot C)\geq a_{i_0}(G_{i_0}\cdot C).$$ We deduce that $a_{i_0}$ cannot be zero (we necessarily have $a_{i_0}>0$ and $G_{i_0}\cdot C<0$).
\end{proof}

\begin{rem} 
In the proof of Theorem~\ref{th:QGorDiv}, we actually prove that $E:=K_X-\phi^*K_Y$ is an exceptional and effective $\Qbb$-Cartier divisor of $X$ such that $E\cdot C<0$ for any curve $C$ of $X$ contracted by $\phi$. Also, $E=\sum_{i\in I}a_iE_i$ where for any $i\in I$, $a_i$ is a positive rational number and where $\{E_i\,\mid\,i\in I\}$ is the set of exceptional divisors of $\phi$.
\end{rem}

\section{Flips}\label{section3}

In this section we interest at the existence of flips (see Definition~\ref{def:QGorflip}).

Note that, if $\phi:X\longrightarrow Y$ is a birational contraction of a ray of $NE(X)_{K_X<0}$ and $m$ is a positive integer such that $mK_X$ is Cartier, then by Lemma~\ref{lem:key} applied to $D=mK_X$ (and $D_Y=mK_Y$), we have 
$$\oplus_{l\geq 0}\mathcal{O}_Y(lmK_Y)=\oplus_{l\geq 0}\phi_*\mathcal{O}_X(lmK_X).$$
We denote by $\mathcal{A}$ this sheaf of $\mathcal{O}_Y$-algebras.

Then, we get an analogue result as in the classical MMP.

\begin{teo}\label{th:QGorFlip}\cite{Shokurov}
There exists a flip of $\phi$ if and only if $\mathcal{A}$ is finitely generated as sheaf of $\mathcal{O}_Y$-algebras. In that case, the flip is unique, given by $\phi^+:\,X^+:=\rm{Proj}(\mathcal{A})\longrightarrow Y$.

Moreover, if $X$ has terminal singularities, then $X^+$ has also terminal singularities.
\end{teo}
 
\begin{proof}
Suppose that there exists a flip $\phi^+:\,X^+\longrightarrow Y$ of $\phi$.
Since $\phi^+$ is a small contraction,  for any $l\geq 0$ we have $\phi^+_*(\mathcal{O}_{X^+}(lK_{X^+}))=\mathcal{O}_Y(lK_{Y})$. In particular, $\mathcal{A}=\oplus_{l\geq 0}\phi_*\mathcal{O}_{X^+}(lmK_{X^+})$. But $mK_{X^+}$ is $\phi^+$-ample (because for any contracted curve $C$ of $X^+$, $K_{X^+}\cdot C>0$), hence $\rm{Proj}(\mathcal{A})$ is a finitely generated sheaf of $\mathcal{O}_Y$-algebra and $X^+=\rm{Proj}(\mathcal{A})$.\\

Suppose now that $\mathcal{A}$ is finitely generated as sheaf of $\mathcal{O}_Y$-algebras, and define $X^+:=\rm{Proj}(\mathcal{A})$. Denote by $\phi^+$ the corresponding morphism $\rm{Proj}(\mathcal{A})\longrightarrow Y$. Then $X^+$ is clearly a normal variety. 

For the rest of the proof, we choose $m$ sufficiently large such that $\mathcal{A}$ is generated by $\phi_*\mathcal{O}_X(mK_X)$, which equals $\mathcal{O}_Y(mK_Y))$; it does not change $\rm{Proj}(\mathcal{A})$. Then we denote by $\mathcal{O}_{X^+}(1)$ the $\phi^+$-very ample invertible sheaf on $X^+$ such that $\phi^+_*\mathcal{O}_{X^+}(1)=\phi_*\mathcal{O}_X(mK_X)$.

We know prove, by contradiction, that $\phi^+$ does not contract a divisor. Let $E$ be an irreducible exceptional divisor of $\phi^+$. We get the following exact sequence $$0\longrightarrow \mathcal{O}_{X^+}\longrightarrow\mathcal{O}_{X^+}(E)\longrightarrow\rm{Coker}\longrightarrow 0$$
where $\rm{Coker}$ cannot be zero. For a positive integer $l$, we apply the functor $\phi^+_*(-\otimes\mathcal{O}_{X^+}(l))$ to this sequence, to obtain $$0\longrightarrow \phi^+_*\mathcal{O}_{X^+}(l)\longrightarrow\phi^+_*(\mathcal{O}_{X^+}(E)\otimes\mathcal{O}_{X^+}(l))\longrightarrow\phi^+_*(\rm{Coker}\otimes\mathcal{O}_{X^+}(l))\longrightarrow R^1\phi^+_*\mathcal{O}_{X^+}(l)\longrightarrow\cdots$$

Choose $l$ sufficiently large such that $R^1\phi^+_*\mathcal{O}_{X^+}(l)=0$ and such that the map $(\phi^+)^*\phi^+_*(\rm{Coker}\otimes\mathcal{O}_{X^+}(l))\longrightarrow\rm{Coker}\otimes\mathcal{O}_{X^+}(l)$ is surjective (it is possible because $\mathcal{O}_{X^+}(1)$ is $\phi^+$-ample). Then $$0\longrightarrow \phi^+_*\mathcal{O}_{X^+}(l)\longrightarrow\phi^+_*(\mathcal{O}_{X^+}(E)\otimes\mathcal{O}_{X^+}(l))\longrightarrow\phi^+_*(\rm{Coker}\otimes\mathcal{O}_{X^+}(l))\longrightarrow 0.$$
We claim that the first map of the above sequence is surjective. Indeed, since $E$ is exceptional, for any open set $\mathcal{U}$ of $Y$, we have the following commutative diagram:

$$\xymatrix{
\mathcal{O}_X(lmK_Y)(\mathcal{U})\ar@{=}[r]\ar@{=}[d]  & \mathcal{O}_X(lmK_Y)(\mathcal{U}\backslash\phi^+(E)) \ar@{=}[d]\\
\phi^+_*\mathcal{O}_{X^+}(l)(\mathcal{U}) \ar@{^{(}->}[r] \ar@{^{(}->}[d] & \phi^+_*\mathcal{O}_{X^+}(l)(\mathcal{U}\backslash\phi^+(E)) \ar@{=}[d] \\
\phi^+_*(\mathcal{O}_{X^+}(E)\otimes\mathcal{O}_{X^+}(l))(\mathcal{U}) \ar@{^{(}->}[r] & \phi^+_*(\mathcal{O}_{X^+}(E)\otimes\mathcal{O}_{X^+}(l))(\mathcal{U}\backslash\phi^+(E)),
}$$
where all inclusions have to be equalities. 

Hence, $\phi^+_*(\rm{Coker}\otimes\mathcal{O}_{X^+}(l))=0$. But  $\phi^{+*}\phi^+_*(\rm{Coker}\otimes\mathcal{O}_{X^+}(l))$ surjects to $\rm{Coker}\otimes\mathcal{O}_{X^+}(l)$, so $\rm{Coker}\otimes\mathcal{O}_{X^+}(l)=0$ and then $\rm{Coker}=0$. We get a contraction.

Now, since $\phi^+$ is a small contraction, we get that $\mathcal{O}(mK_{X^+})$ is isomorphic to  $\mathcal{O}_{X^+}(1)$ so that $X^+$ is clearly $\Qbb$-Gorenstein and $K_{X^+}$ is $\phi^+$-ample (ie for any contracted curve $C$ of $X^+$, $K_{X^+}\cdot C>0$).

Suppose now that $X$ has terminal singularities. We consider a common desingularization of $X$ and $X^+$:
$$\xymatrix{
 & V\ar[ld]_{\sigma}\ar[rd]^{\sigma^+}\\
 X\ar[rd]_{\phi} & &X^+\ar[ld]^{\phi^+}\\
 & Y &
}$$

Since $mK_{X^+}$ is $\phi^+$-ample, we have $$\begin{array}{rcl}
\sigma^{+*}\mathcal{O}_{X^+}(mK_{X^+}) & = & \sigma^{+*}(\Im(\phi^{+*}\phi^+_*\mathcal{O}_{X^+}(mK_{X^+})\longrightarrow\mathcal{O}_{X^+}(mK_{X^+})))\\
 & = & \Im(\sigma^{+*}\phi^{+*}\phi^+_*\mathcal{O}_{X^+}(mK_{X^+})\longrightarrow\sigma^{+*}\mathcal{O}_{X^+}(mK_{X^+}))).
\end{array}$$
But $\sigma^{+*}\phi^{+*}\phi^+_*\mathcal{O}_{X^+}(mK_{X^+})=(\phi^+\circ\sigma^+)^*\mathcal{O}_Y(mK_Y)=(\phi\circ\sigma)^*\phi_*\mathcal{O}_X(mK_X)=
(\phi\circ\sigma)^*(\phi\circ\sigma)_*\sigma^*\mathcal{O}_X(mK_X)$.
Thus, $\sigma^{+*}\mathcal{O}_{X^+}(mK_{X^+})$ is contained in $\sigma^*\mathcal{O}_X(mK_X)$.
In particular, by Lemma~\ref{lem:eff}, the divisor $\sigma^*K_X-\sigma^{+*}K_{X^+}$ is effective.

Hence, the divisor $K_V-\sigma^{+*}K_{X^+}=(K_V-\sigma^*K_X)+(\sigma^*K_X-\sigma^{+*}K_{X^+})$ is effective, and moreover, it has positive coefficient in the irreducible divisors of $V$ contracted by $\sigma$ (because $X$ has terminal singularities). It remains to prove that it has positive coefficient in the irreducible divisors of $V$ contracted by $\sigma^+$.

Let $E$ be an irreducible divisor of $V$ contracted by $\sigma^+$, but not contracted by $\sigma$ (if it exists). The irreducible divisor $\sigma(E)$ of $X$ is contracted by $\phi$. Then there exists a curve $C$  in $E$ that is contracted by $\sigma^+$ but not by $\sigma$. In particular, $\sigma(C)$ is a curve of $X$ contracted by $\phi$.
On the one hand $\sigma^{+*}K_{X^+}\cdot C=0$, and on the other hand $K_X\cdot \sigma(C)<0$ so that $\sigma^*K_X\cdot C<0$. Hence, $(\sigma^*K_X-\sigma^{+*}K_{X^+})\cdot C<0$. We can choose the curve $C$ such that it is contained in no other exceptional divisor of $\sigma$ and $\sigma^+$. Then we conclude as in the proof of Theorem~\ref{th:QGorDiv}, that the coefficient in $E$ of $\sigma^*K_X-\sigma^{+*}K_{X^+}$ is not zero.
\end{proof}

\begin{cor}[Corollary of the proof of Theorem~\ref{th:QGorFlip}]\label{cor:QGorFlip}
Let $\phi^+:\,X^+\longrightarrow Y$ be a flip of an extremal contraction  $\phi:\,X\longrightarrow Y$. Let $V$ be a commun desingularization of $X$ and $X^+$. Denote $\sigma:\,V\longrightarrow X$ and $\sigma^+:\,V\longrightarrow X^+$. 
Then $K_V=\sigma^*K_X+\sum_{i\in I}a_iE_i=\sigma^{+*}K_{X^+}+\sum_{i\in I}a_i^+E_i$, such that, for any $i\in I$, $a_i^+\geq a_i$. Moreover, there exists $i_0\in I$ with $a_{i_0}^+>a_{i_0}$.
\end{cor}

Note that the set $\{E_i\,\mid\,i\in I\}$ is the union of the sets of irreducible exceptional divisors of $\sigma$ and $\sigma^+$, the hypothesis of Corollary~\ref{cor:QGorFlip} implies that for any $i\in I$, $a_i^+>0$, $a_i\geq 0$ and $a_i=0$ if and only if $E_i$ is not an exceptional divisor of $\sigma$.

\begin{proof}
It is enough to prove that, the effective divisor $\sigma^*K_X-\sigma^{+*}K_{X^+}$ is not zero. 
Let $C$ be a curve of $V$ contracted by $\phi\circ\sigma=\phi^+\circ\sigma^+$ but not by $\sigma$. Then, by hypothesis on $\phi$, $\sigma^*K_X\cdot C<0$. If $C$ is contracted by $\phi^+$, we have $\sigma^{+*}K_{X^+}\cdot C=0$, and if not, we have $\sigma^{+*}K_{X^+}\cdot C>0$. In any cases,  $\sigma^{+*}K_{X^+}\cdot C\geq 0$ so that $(\sigma^*K_X-\sigma^{+*}K_{X^+})\cdot C>0$. In particular $\sigma^*K_X-\sigma^{+*}K_{X^+}$ is not zero.
\end{proof}

This corollary will be useful to prove the finiteness of sequences of flips in the family of $\Qbb$-Gorenstein spherical varieties.

\section{An example}\label{section4}

Here, we give an example of the $\Qbb$-Gorenstein MMP for a 3-dimensional toric variety. Se also \cite{Fujino3} for other toric examples.
For the basics of theory of toric varieties, the reader can see \cite{fulton} or \cite{oda}.

$\bullet$ In $\Zbb^3\subset\Qbb^3$, we consider the six following vectors:
$$\begin{array}{rclrclrcl}
e_1 &=& (-1,-1,1),\,&
e_2 &=& (1,-1,1),\,&
e_3 &=& (1,1,2),\,\\
e_4 &=& (-1,1,2),\,&
e_5 &=& (0,1,1),\,&
e_6 &=& (0,0,-1).
\end{array}$$

We denote by $\mathcal{C}(e_{i_1},\dots,e_{i_k})$ the cone in $\Qbb^3$ generated by $e_{i_1},\dots,e_{i_k}$. Then $$\mathbb{F}:=\{\mathcal{C}(e_1,e_2,e_3,e_4),\mathcal{C}(e_3,e_4,e_5),\mathcal{C}(e_1,e_2,e_6),\mathcal{C}(e_1,e_4,e_6),\mathcal{C}(e_2,e_3,e_6),\mathcal{C}(e_3,e_5,e_6),\mathcal{C}(e_4,e_5,e_6)\}$$ is a complete fan of $\Zbb^3$.
We denote by $X$ the toric variety of fan $\mathbb{F}$.

Note that since $\mathbb{F}$ contains a non-simplicial cone, the variety $X$ is not $\Qbb$-factorial.

Denote by $X_1$, $X_2$, $X_3$, $X_4$, $X_5$ and $X_6$ the $(\Cbb^*)^3$-stable irreducible divisors respectively associated to the rays of $\mathbb{F}$ generated by $e_1$, $e_2$, $e_3$, $e_4$, $e_5$ and $e_6$. We can compute that $$\rm{Pic}(X)_\Qbb:=\bigslant{\left\lbrace\sum_{i=1}^6a_iX_i\,\mid\,
\begin{array}{c} a_1-a_2+a_3-a_4=0\\
a_1,\,a_2,\,a_3,\,a_4,\,a_5,\,a_6\in\Qbb
\end{array}
\right\rbrace}{\left\langle \begin{array}{c}X_1-X_2-X_3+X_4,\\X_1+X_2-X_3-X_4-X_5,\\X_1+X_2+2X_3+2X_4+X_5-X_6\end{array}\right\rangle}.$$ In particular, $-K_X=\sum_{i=1}^6X_i$ is $\Qbb$-Cartier.

We know that $NE(X)$ is generated by the classes of $(\Cbb^*)^3$-stable (and rational) curves $C_{12}$, $C_{14}$, $C_{16}$,  $C_{23}$,  $C_{26}$, $C_{34}$ ,$C_{35}$, $C_{36}$,  $C_{45}$, $C_{46}$, $C_{56}$, where $C_{ij}$ denotes the $(\Cbb^*)^3$-stable curve associated to the 2-codimensional cone of $\mathbb{F}$ generated by $e_i$ and $e_j$.

We choose the basis $(X_1+X_2,X_1-X_3)$ of $\rm{Pic}(X)_\Qbb$ and we compute the classes of these curves in the corresponding dual basis: 
$$\begin{array}{rclrclrclrcl}
[C_{12}] &=& (\frac{1}{3},0), &
[C_{14}] &=& (\frac{1}{6},0), &
[C_{16}] &=& (\frac{1}{2},0), &
[C_{23}] &=& (\frac{1}{6},0), \\ 
~[C_{26}] &=& (\frac{1}{2},0), &
[C_{34}] &=& (\frac{1}{3},\frac{1}{2}), &
[C_{35}] &=& (0,-\frac{1}{2}), &
[C_{36}] &=& (\frac{1}{2},\frac{1}{2}), \\
~[C_{45}] &=& (0,-\frac{1}{2}), &
[C_{46}] &=& (\frac{1}{2},\frac{1}{2}), &
[C_{56}] &=& (0,-1). &
\end{array}$$

In particular,  the cone $NE(X)$ is generated by $(2,3)$ and $(0,-1)$, ie by $[C_{34}]$ and $[C_{35}]$.

Note also that $-K_X$ is linearly equivalent to $5(X_1+X_2)-2(X_1-X_3)$ so that it is positive on all effective curves. Hence, in order to run the ($\Qbb$-Gorenstein) MMP, we need to choose one of the two extremal rays in $NE(X)_{K_X<0}$.\\

$\bullet$ First, consider the contraction $\phi:\,X\longrightarrow Y$ of the extremal ray generated by $[C_{35}]$. We remark that this ray contains the classes of $C_{35}$, $C_{45}$, and $C_{56}$. Then $Y$ is the 3-dimensional toric variety whose fan is $$\mathbb{F}_Y:=\{\mathcal{C}(e_1,e_2,e_3,e_4),\mathcal{C}(e_3,e_4,e_6),\mathcal{C}(e_1,e_2,e_6),\mathcal{C}(e_1,e_4,e_6),\mathcal{C}(e_2,e_3,e_6)\}.$$ 
We still denote by $X_1$, $X_2$, $X_3$, $X_4$ and $X_6$ the $(\Cbb^*)^3$-stable irreducible divisors of $Y$. And we compute that $$\rm{Pic}(Y)_\Qbb:=\bigslant{\left\lbrace\sum_{i=1}^4a_iX_i+a_6X_6\,\mid\,
\begin{array}{c} a_1-a_2+a_3-a_4=0\\
a_1,\,a_2,\,a_3,\,a_4,\,a_6\in\Qbb
\end{array}
\right\rbrace}{\left\langle \begin{array}{c}X_1-X_2-X_3+X_4,\\X_1+X_2-X_3-X_4,\\X_1+X_2+2X_3+2X_4-X_6\end{array}\right\rangle}.$$ In particular, $-K_Y=\sum_{i=1}^4X_i+X_6$ is $\Qbb$-Cartier (and generates $\rm{Pic}(Y)_\Qbb$). The contraction $\phi$ is a $\Qbb$-Gorenstein divisorial contraction. \\

$\bullet$ Now, consider the contraction $\phi:\,X\longrightarrow Y$ of the extremal ray generated by $[C_{34}]$. Then $Y$ is the 3-dimensional toric variety whose fan is $$\mathbb{F}_Y:=\{\mathcal{C}(e_1,e_2,e_3,e_4,e_5),\mathcal{C}(e_1,e_2,e_6),\mathcal{C}(e_1,e_4,e_6),\mathcal{C}(e_2,e_3,e_6),\mathcal{C}(e_3,e_5,e_6),\mathcal{C}(e_4,e_5,e_6)\}.$$
We still denote by $X_1$, $X_2$, $X_3$, $X_4$, $X_5$ and $X_6$ the $(\Cbb^*)^3$-stable irreducible divisors of $Y$.
And we compute that $$\rm{Pic}(Y)_\Qbb:=\bigslant{\left\lbrace\sum_{i=1}^6a_iX_i\,\mid\,
\begin{array}{c} a_1-a_2+a_3-a_4=0\\
a_1-3a_2+4a_3-6a_5=0\\
a_1,\,a_2,\,a_3,\,a_4,\,a_5,\,a_6\in\Qbb
\end{array}
\right\rbrace}{\left\langle \begin{array}{c}X_1-X_2-X_3+X_4,\\X_1+X_2-X_3-X_4-X_5,\\X_1+X_2+2X_3+2X_4+X_5-X_6\end{array}\right\rangle}.$$ In particular, $-K_Y=\sum_{i=1}^6X_i$ is not $\Qbb$-Cartier.

The flip of $\phi$ is given by the $\Qbb$-factorial 3-dimensional toric variety $X^+$ whose fan is \begin{multline*}\mathbb{F}_{X^+}:=\{\mathcal{C}(e_1,e_2,e_3), \mathcal{C}(e_1,e_3,e_5),\mathcal{C}(e_1,e_4,e_5), \\  \mathcal{C}(e_1,e_2,e_6),\mathcal{C}(e_1,e_4,e_6),\mathcal{C}(e_2,e_3,e_6),\mathcal{C}(e_3,e_5,e_6),\mathcal{C}(e_4,e_5,e_6)\},\end{multline*} and the $(\Cbb^*)^3$-equivariant map $\phi^+:\,X^+\longrightarrow Y$.
We still denote by $X_1$, $X_2$, $X_3$, $X_4$, $X_5$ and $X_6$ the $(\Cbb^*)^3$-stable irreducible divisors of $X^+$.
In the basis dual to $(X_1,X_2,X_3)$, the classes of the $(\Cbb^*)^3$-stable curves of $X^+$ are 
$$\begin{array}{rclrclrclrcl}
[C_{12}] &=& (\frac{1}{3},0,\frac{1}{3}), &
[C_{13}] &=& (-\frac{1}{2},\frac{1}{2},-\frac{2}{3}), &
[C_{14}] &=& (\frac{1}{2},0,0), &
[C_{15}] &=& (-\frac{1}{3},0,\frac{1}{3}), \\
~[C_{16}] &=& (0,\frac{1}{2},0), &
[C_{23}] &=& (\frac{1}{6},0,\frac{1}{6}), &
[C_{26}] &=& (\frac{1}{2},0,\frac{1}{2}), &
[C_{35}] &=& (\frac{1}{3},0,\frac{1}{3}), \\
~[C_{36}] &=& (0,\frac{1}{2},-\frac{1}{2}), &
[C_{45}] &=& (1,0,0), &
[C_{46}] &=& (\frac{1}{2},0,0), &
[C_{56}] &=& (0,0,1).

\end{array}$$

 We deduce that the cone $NE(X^+)$ is generated by $(-3,3,-4)$, $(1,0,0)$ and $(-1,0,1)$. Moreover, since $-K_{X^+}$ is linearly equivalent to $3X_1+5X_2+2X_3$, there are two extremal rays in $NE(X^+)_{K_{X^+}>0}$, respectively generated by $[C_{13}]$ and $[C_{15}]$. The map $\phi^+$ is the contraction of the 2-dimensional face of $NE(X)$ generated by $[C_{13}]$ and $[C_{15}]$. \\
 
For more examples, we refer to \cite{MMPhoro} where a flip of a contraction that contracts a divisor is given.

\section{Log $\Qbb$-Gorenstein MMP}\label{section5}

As for $\Qbb$-factorial varieties, we can run a $\Qbb$-Gorenstein MMP for klt pairs $(X,D)$.

Let $X$ be a normal variety and let $D$ be an effective $\Qbb$-divisor such that $K_X+D$ is $\Qbb$-Cartier.

\begin{defi}
The pair $(X,D)$ is said to be klt (Kawamata log terminal) if $D$ there exists a desingularization  $\sigma:\,V\longrightarrow X$ of $X$ such that $K_V=\sigma^*(K_X+D)+\sum_{i\in I}a_iE_i$ where the $E_i$'s are irreducible divisors of $V$ and for any $i\in I$, $a_i>-1$.
\end{defi}

\begin{rem}
\begin{enumerate}
\item If a pair $(X,D)$ is klt, then the above property is true for every desingularization of $X$.
\item The condition "for any $i\in I$, $a_i>-1$" can be replaced by: $\lfloor D\rfloor=0$ and for any $i\in I$ such that $E_i$ is exceptional for $\sigma$, $a_i>-1$.
\end{enumerate}
\end{rem}

Suppose now that $(X,D)$ is klt.

By \cite{KMM}, the Contraction Theorem and the Cone Theorem are still valid. In particular, for any ray $R$ of $NE(X)_{K_X+D<0}$, there exists a unique contraction $\phi:\,X\longrightarrow Y$, such that, for any curve $C$ of $X$, $\phi(C)$ is a point if and only if the class of $C$ is in $R$. Moreover, we have equivalent results of Theorems~\ref{th:QGorDiv} and~\ref{th:QGorFlip}.

The following theorem was already given by O.~Fujino \cite{Fujino2} and generalized to log-canonical pairs in \cite{Fujino3}.

\begin{teo} We denote by $D_Y$ the $\Qbb$-divisor $\phi_*D$ of $Y$. And we fix $m\geq 1$ such that $m(K_X+D)$ is Cartier.
\begin{enumerate}
\item If $K_Y+D_Y$ is $\Qbb$-Cartier, then the pair $(Y,D_Y)$ is klt and $E:=K_X+D-\phi^*(K_Y+D_Y)$ is an exceptional and effective $\Qbb$-Cartier divisor such that $E.C<0$ for any curve $C$ of $X$ contracted by $\phi$.

\item The sheaf of $\mathcal{O}_Y$-algebras $\mathcal{A}:=\bigoplus_{l\geq 0}\phi_*\mathcal{O}_X(lm(K_X+D))$ equals $\bigoplus_{l\geq 0}\mathcal{O}_Y(lm(K_Y+D_Y))$.

\item If $K_Y+D_Y$ is not $\Qbb$-Cartier, $\mathcal{A}$ is finitely generated if and only if there exists a small contraction $\phi^+:\,X^+\longrightarrow Y$ such that the pair $(X^+,(\phi^+_*)^{-1}D_Y)$ is klt and for any curve $C^+$ of $X^+$ contracted by $\phi^+$, $(K_{X^+}+(\phi^+_*)^{-1}D_Y)\cdot C^+>0$. In that case, $X^+$ is $\rm{Proj}(\mathcal{A})$ over $Y$.
\end{enumerate}
\end{teo}

The proof is very similar to the proofs of Theorems~\ref{th:QGorDiv} and~\ref{th:QGorFlip}. The key of the proof of (1) is that, as in Corollary~\ref{cor:keyeff}, the divisor $K_X+D-\phi^*(K_Y+D_Y)$ is effective. Lemma~\ref{lem:key}, applied to $m(K_X+D)$ (and $(m(K_Y+D_Y)$) directly gives (2). And to prove (3), we do the same proof as in Theorem~\ref{th:QGorFlip}, by replacing $K_X$ by $K_X+D$, $K_Y$ by $K_Y+D_Y$ and $K_{X^+}$ by $K_{X^+}+(\phi^+_*)^{-1}D_Y$ (excepting the last paragraph, which is not necessary).

\section{Open questions}\label{section6}

The same questions as in $\Qbb$-factorial MMP can be done.\\

\noindent{\bf Question~1.} Do  flips always exist? Or equivalently, in the case where $Y$ is not $\Qbb$-Gorenstein, is $\mathcal{A}:=\bigoplus_{l\geq 0}\phi_*\mathcal{O}_X(lmK_X)=\bigoplus_{l\geq 0}\mathcal{O}_Y(lmK_Y)$ finitely generated as sheaf of $\mathcal{O}_Y$-algebras?\\

\noindent{\bf Question~2.} Are sequences of flips always finite?\\

We can answer positively these two questions in the case of spherical varieties.\\

Let $G$ be a connected reductive algebraic group (over $\Cbb$). A normal $G$-variety is spherical if there exists an open orbit in $X$ under the action of a Borel subgroup of $G$.

Let $H$ be a spherical subgroup of $G$ (ie such that there exists a Borel subgroup of $G$ satisfying that $BH$ is open in $G$). A $G/H$-embedding is a normal $G$-variety containing an open $G$-orbit isomorphic to $G/H$. (A $G/H$-embedding is a spherical $G$-variety, and inversely a spherical $G$-variety is a $G/H$-embedding for some spherical subgroup $H$ of $G$.)

\begin{prop} \label{prop:spher} 
\begin{enumerate}
\item (\cite[Lemme 4.3]{brionmori}) Let $X$ and $Y$ be two spherical $G$-varieties and let $\phi$ be a proper $G$-equivariant morphism. Then for any Cartier divisor $D$ of $X$, the $\mathcal{O}_Y$-algebra $\bigoplus_{l\geq 0}\phi_*\mathcal{O}_X(lD)$ is finitely generated.
\item Let $X$ be a $G/H$-embeddings (ie a normal $G$-variety containing an open $G$-orbit isomorphic to $G/H$). There are only finitely many varieties $Z$ that can be obtained from $X$ by flips (and these varieties are still $G/H$-embeddings). Moreover, there exists a commun desingularization for all these varieties, ie there exist a smooth $G/H$-embedding $V$ and birational proper $G$-equivariant morphisms $\sigma_Z:\,V\longrightarrow Z$, for any $Z$.
\end{enumerate}
\end{prop} 

The second part of the proposition is a consequence of the classification of $G/H$-embeddings in terms of colored fans and the fact that a flip adds no divisor.\\

Hence, in the family of spherical varieties, we immediately get the existence of flips, and the finiteness of sequences of flips is a consequence of {\it 2} of Proposition~\ref{prop:spher} and Corollary~\ref{cor:QGorFlip}.\\

\bibliographystyle{amsalpha}

\providecommand{\bysame}{\leavevmode\hbox to3em{\hrulefill}\thinspace}
\providecommand{\MR}{\relax\ifhmode\unskip\space\fi MR }
\providecommand{\MRhref}[2]{%
  \href{http://www.ams.org/mathscinet-getitem?mr=#1}{#2}
}
\providecommand{\href}[2]{#2}

\end{document}